\documentclass[a4paper,11pt]{article} 
\usepackage[top=2.5cm, bottom=2.5cm, left=2.5cm, right=2.5cm]{geometry}
\usepackage[english]{babel}
\usepackage[T1]{fontenc}
\usepackage[utf8]{inputenc}
\usepackage{amssymb}
\usepackage{lmodern}
\usepackage{microtype}
\usepackage{hyperref}
\usepackage{graphicx}
\usepackage{amsmath}
\usepackage{amsmath,amsfonts}
\usepackage{tikz}
\usepackage{array}

\usepackage{arydshln}
\usepackage{rotating}
\usepackage[justification=centering]{caption}
\newcolumntype{P}[1]{>{\centering\arraybackslash}p{#1}}
\usepackage{graphicx}
\usepackage{subcaption}
\usepackage{indentfirst}
\usepackage{mathtools}

\usepackage{float}
\usepackage{amsmath,amssymb,amsthm}
\providecommand{\keywords}[1]{\textbf{\textit{Keywords:}} #1}

\newtheorem{theorem}{Theorem}[section]
\newtheorem{cor}{Corollary}[section]
\newtheorem{lemme}{Lemma}[section]

\newtheorem{definition}{Definition}[section]

\newtheorem{proposition}{Proposition}[section]

\newcommand{\cube}[3]{
			\fill[color = black!20] (#1,#2,#3) +(0,1,0) -- +(0,1,-1)
 -- +(1,1,-1) -- +(1,1,0) -- cycle;
			\fill[color = black!30] (#1,#2,#3) -- +(0,1,0) -- +(1,1,
0) -- +(1,0,0) -- cycle;
			\fill[color = black!50] (#1,#2,#3) +(1,0,0) -- +(1,1,0) 
-- +(1,1,-1) -- +(1,0,-1) -- cycle;
			\draw[color = black!80] (#1,#2,#3) -- +(0,1) -- +(1,1) -
- +(1,0) -- cycle;
.025,.94,0) -- +(0,1,-1) -- +(1,1,-1) -- +(1,0,-1) -- +(1,0,0);
			\draw[color = black!80] (#1,#2,#3) + (0,1,0) -- +(0,1,-1
) -- +(1,1,-1) -- +(1,0,-1) -- +(1,0,0);
			\draw[color = black!80] (#1,#2,#3) +(1,1,0) -- +(1,1,-1)
;
}

\newcommand{\cell}[2]{
			
			\fill[color = black!30] (#1,#2) -- +(0,1) -- +(1,1) -- +(1,0) -- cycle;
					\draw[color = black!80] (#1,#2) -- +(0,1) -- +(1,1) -- +(1,0) -- cycle;
;
}
\newcommand{\freeza}[3]{
			
			\fill[color = black!30] (#1,#2,#3) -- +(0,1,0) -- +(1,1,0) -- +(1,0,0) -- cycle;
					\draw[color = black!80] (#1,#2,#3) -- +(0,1,0) -- +(1,1,0) -- +(1,0,0) -- cycle;
;
}
\newcommand{\buu}[3]{
			
			\fill[color = black!30] (#1,#2,#3) -- +(0,0,1) -- +(1,0,1) -- +(1,0,0) -- cycle;
					\draw[color = black!80](#1,#2,#3) -- +(0,0,1) -- +(1,0,1) -- +(1,0,0) -- cycle;
;
}

\title{\textbf{Plateau Polycubes and Lateral Area}}
\author{Abderrahim Arabi\\
\small USTHB, Faculty of Mathematics\\
\small RECITS Laboratory\\
\small BP 32, El Alia 16111, Bab Ezzouar\\
\small Algiers, Algeria\\
\small\tt rarabi@usthb.dz 
\and
Hacène Belbachir\\
\small USTHB, Faculty of Mathematics\\
\small RECITS Laboratory\\
\small BP 32, El Alia 16111, Bab Ezzouar\\
\small Algiers, Algeria\\
\small\tt hbelbachir@usthb.dz \\ 
\and
Jean-Philippe Dubernard\footnote{The paper is partially supported by the ERDF/GRR project MOUSTIC.}\\
\small University of Rouen-Normandie, Faculty of Science and Technique\\ 
\small LITIS Laboratory\\
\small Avenue de l’université 76800 Saint-Étienne-du-Rouvray\\
\small Rouen, France\\
\small\tt jean-philippe.dubernard@univ-rouen.fr
}
\date{}
\begin{document}
\maketitle
\begin{abstract}
In this paper, we enumerate two families of polycubes, the directed plateau polycubes and the plateau polycubes, with respect to the width and a new parameter, the Lateral Area. We give an explicit formula and the generating function for each of the two families of polycubes. Moreover, some asymptotic results about plateau polycubes are provided. We also establish results concerning the enumeration of column-convex polyominoes that are useful to get asymptotic results of polycubes. 
\end{abstract}
\section{Introduction}
\label{Intr}
\indent 
In the cartesian plane $\mathbb{Z}^2$, a polyomino is a finite union of cells (unit squares), connected by their edges, without a cut point and defined up to a translation, for the concept see \cite{ref21}.
Polyominoes appear in statistical physics in percolation theory by the appellation of animals. They are obtained by replacing each cell by its center \cite{ref12}. Enumeration of polyominoes in general case is still an open problem. Some algorithms were made and the number of polyominoes with $n$ cells is known up to $n=56$, \cite{ref7}. However, exact enumeration exists for several families of polyominoes, (see for instance \cite{ref8, ref9}). 
\begin{figure}[H]
\centering
\begin{tikzpicture}[scale=0.5]
\fill[color=gray!40](0,1)--(1,1)--(1,2)--(0,2);
\fill[color=gray!40](0,2)--(1,2)--(1,3)--(0,3);
\fill[color=gray!40](0,3)--(1,3)--(1,4)--(0,4);

\fill[color=gray!40](1,2)--(2,2)--(2,3)--(1,3);
\fill[color=gray!40](1,4)--(2,4)--(2,5)--(1,5);
\fill[color=gray!40](1,5)--(2,5)--(2,6)--(1,6);
\fill[color=gray!40](1,6)--(2,6)--(2,7)--(1,7);

\fill[color=gray!40](2,0)--(3,0)--(3,1)--(2,1);
\fill[color=gray!40](2,1)--(3,1)--(3,2)--(2,2);
\fill[color=gray!40](2,2)--(3,2)--(3,3)--(2,3);
\fill[color=gray!40](2,4)--(3,4)--(3,5)--(2,5);
\fill[color=gray!40](2,6)--(3,6)--(3,7)--(2,7);

\fill[color=gray!40](3,1)--(4,1)--(4,2)--(3,2);
\fill[color=gray!40](3,2)--(4,2)--(4,3)--(3,3);
\fill[color=gray!40](3,3)--(4,3)--(4,4)--(3,4);
\fill[color=gray!40](3,4)--(4,4)--(4,5)--(3,5);

\fill[color=gray!40](4,1)--(5,1)--(5,2)--(4,2);
\fill[color=gray!40](4,3)--(5,3)--(5,4)--(4,4);

\fill[color=gray!40](5,1)--(6,1)--(6,2)--(5,2);
\fill[color=gray!40](5,2)--(6,2)--(6,3)--(5,3);
\fill[color=gray!40](5,3)--(6,3)--(6,4)--(5,4);
\fill[color=gray!40](5,4)--(6,4)--(6,5)--(5,5);

\fill[color=gray!40](6,3)--(7,3)--(7,4)--(6,4);

\fill[color=gray!40](7,2)--(8,2)--(8,3)--(7,3);
\fill[color=gray!40](7,3)--(8,3)--(8,4)--(7,4);

\draw (0,1)--(1,1)--(1,2)--(0,2);
\draw (0,2)--(1,2)--(1,3)--(0,3);
\draw (0,3)--(1,3)--(1,4)--(0,4);
\draw (0,1)--(0,4);

\draw(1,2)--(2,2)--(2,3)--(1,3);
\draw(1,4)--(2,4)--(2,5)--(1,5);
\draw(1,5)--(2,5)--(2,6)--(1,6);
\draw(1,6)--(2,6)--(2,7)--(1,7);
\draw (1,4)--(1,7);

\draw(2,0)--(3,0)--(3,1)--(2,1);
\draw(2,1)--(3,1)--(3,2)--(2,2);
\draw(2,2)--(3,2)--(3,3)--(2,3);
\draw(2,4)--(3,4)--(3,5)--(2,5);
\draw(2,6)--(3,6)--(3,7)--(2,7);
\draw(2,0)--(2,2);

\draw(3,1)--(4,1)--(4,2)--(3,2);
\draw(3,2)--(4,2)--(4,3)--(3,3);
\draw(3,3)--(4,3)--(4,4)--(3,4);
\draw(3,4)--(4,4)--(4,5)--(3,5);
\draw(3,3)--(3,4);

\draw(4,1)--(5,1)--(5,2)--(4,2);
\draw(4,3)--(5,3)--(5,4)--(4,4);

\draw(5,1)--(6,1)--(6,2)--(5,2);
\draw(5,2)--(6,2)--(6,3)--(5,3);
\draw(5,3)--(6,3)--(6,4)--(5,4);
\draw(5,4)--(6,4)--(6,5)--(5,5);
\draw(5,2)--(5,5);

\draw(6,3)--(7,3)--(7,4)--(6,4);

\draw(7,2)--(8,2)--(8,3)--(7,3);
\draw(7,3)--(8,3)--(8,4)--(7,4);
\draw(7,2)--(7,3);

\end{tikzpicture}
\label{f2}
\caption{Example of a polyomino.}
\end{figure}
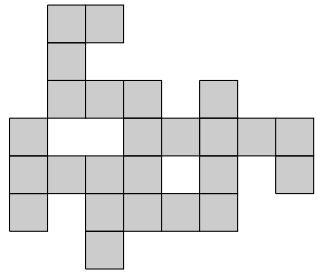

Polycubes are the equivalent of polyominoes in dimension three \cite{ref6}.
In the Cartesian plane $\mathbb{Z}^3$, we define a cell as a unit cube. A polycube is a finite union of cells, connected by their faces and defined up to translation. As polyominoes, polycubes appear in the phenomenon of percolation \cite{ref22}. There is no analytic formula for the number of polycubes of $n$ cells. In $1971$, Lunnon computed the first values up to $n=6$ \cite{ref6}. In $2008$, Aleksandrowicz and Barquet founded the values up to $18$ \cite{ref4} and more recently Luther and Mertens gave the number of polycubes up to $19$ \cite{ref13}. The $d$-dimensional-polycubes are an extension of the notion of polycube to a dimension $d\geq4$. In this case, a cell is a unit hypercube. They are used in an efficient model for real time validation \cite{ref15} and in representation of finite geometric languages \cite{ref16}. 
A few classes of polycubes were studied. Among its, let us cite, for instance, the plane partitions \cite{ref14}, the directed plateau polycubes and also some asymptotic results concerning the parallelograms polycubes that were made according to the number of cells \cite{ref1}.\\
There exist two methods to enumerate families of polycubes. The main one is the generic method \cite{ref3} and the other one is based on the use of the Dirichlet convolution \cite{ref2}.\\ 

Directed plateau polycubes were enumerated according to the volume and the width \cite{ref1}. In this paper we introduce a new parameter, \textit{the lateral area}. As no exact formula exists concerning the enumeration of specific families of polycubes according to the area, the parameter lateral area is still interesting. Indeed, it allows a certain approximation of the area. Using this parameter and the width, we enumerate the directed plateau polycubes and the plateau polycubes. We also give some asymptotic results for this family.
In the next section, we recall some definitions and properties of polyominoes and their extensions to polycubes. Two families of polycubes are enumerated according to the lateral area, the directed plateau polycubes in Section \ref{EDP} and the plateau polycubes in Section \ref{EP}. In Section \ref{AR}, asymptotic results are given for column-convex polyominoes and plateau polycubes.    

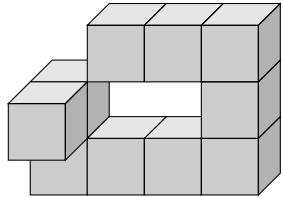
\begin{figure}[H]
\centering
\begin{tikzpicture}[scale=0.75]
\fill[color=gray!60] (1,1,1)--(1,1,0)--(1,2,0)--(1,2,1)--cycle;
\fill[color=gray!20] (0,2,1)--(0,2,0)--(1,2,0)--(1,2,1)--cycle;
\fill[color=gray!20] (2,1,1)--(2,1,0)--(3,1,0)--(3,1,1)--cycle;

\fill[color=gray!40] (0,0,1)--(1,0,1)--(1,1,1)--(0,1,1)--cycle;

\draw (1,1,1)--(1,1,0)--(1,2,0)--(1,2,1)--cycle;
\draw (0,2,1)--(0,2,0)--(1,2,0)--(1,2,1)--cycle;
\draw (2,1,1)--(2,1,0)--(3,1,0)--(3,1,1)--cycle;

\draw  (0,0,1)--(1,0,1)--(1,1,1)--(0,1,1)--cycle;

\fill[color=gray!40] (1,0,1)--(2,0,1)--(2,1,1)--(1,1,1)--cycle;
\fill[color=gray!40] (2,0,1)--(3,0,1)--(3,1,1)--(2,1,1)--cycle;
\fill[color=gray!40] (3,0,1)--(4,0,1)--(4,1,1)--(3,1,1)--cycle;
\fill[color=gray!40] (3,1,1)--(4,1,1)--(4,2,1)--(3,2,1)--cycle;
\fill[color=gray!40] (3,2,1)--(4,2,1)--(4,3,1)--(3,3,1)--cycle;
\fill[color=gray!40] (2,2,1)--(3,2,1)--(3,3,1)--(2,3,1)--cycle;
\fill[color=gray!40] (1,2,1)--(2,2,1)--(2,3,1)--(1,3,1)--cycle;
\fill[color=gray!40] (0,1,2)--(1,1,2)--(1,2,2)--(0,2,2)--cycle;

\fill[color=gray!60] (1,1,2)--(1,1,1)--(1,2,1)--(1,2,2)--cycle;
\fill[color=gray!60] (4,1,1)--(4,1,0)--(4,2,0)--(4,2,1)--cycle;
\fill[color=gray!60] (4,0,1)--(4,0,0)--(4,1,0)--(4,1,1)--cycle;
\fill[color=gray!60] (4,2,1)--(4,2,0)--(4,3,0)--(4,3,1)--cycle;

\fill[color=gray!20] (0,2,2)--(0,2,1)--(1,2,1)--(1,2,2)--cycle;
\fill[color=gray!20] (1,1,1)--(1,1,0)--(2,1,0)--(2,1,1)--cycle;
\fill[color=gray!20] (1,3,1)--(1,3,0)--(2,3,0)--(2,3,1)--cycle;
\fill[color=gray!20] (2,3,1)--(2,3,0)--(3,3,0)--(3,3,1)--cycle;
\fill[color=gray!20] (3,3,1)--(3,3,0)--(4,3,0)--(4,3,1)--cycle;


\draw  (1,0,1)--(2,0,1)--(2,1,1)--(1,1,1)--cycle;
\draw  (2,0,1)--(3,0,1)--(3,1,1)--(2,1,1)--cycle;
\draw (3,0,1)--(4,0,1)--(4,1,1)--(3,1,1)--cycle;
\draw  (3,1,1)--(4,1,1)--(4,2,1)--(3,2,1)--cycle;
\draw (3,2,1)--(4,2,1)--(4,3,1)--(3,3,1)--cycle;
\draw  (2,2,1)--(3,2,1)--(3,3,1)--(2,3,1)--cycle;
\draw  (1,2,1)--(2,2,1)--(2,3,1)--(1,3,1)--cycle;
\draw  (0,1,2)--(1,1,2)--(1,2,2)--(0,2,2)--cycle;

\draw (1,1,2)--(1,1,1)--(1,2,1)--(1,2,2)--cycle;
\draw  (4,1,1)--(4,1,0)--(4,2,0)--(4,2,1)--cycle;
\draw  (4,0,1)--(4,0,0)--(4,1,0)--(4,1,1)--cycle;
\draw  (4,2,1)--(4,2,0)--(4,3,0)--(4,3,1)--cycle;

\draw  (0,2,2)--(0,2,1)--(1,2,1)--(1,2,2)--cycle;
\draw  (1,1,1)--(1,1,0)--(2,1,0)--(2,1,1)--cycle;
\draw  (1,3,1)--(1,3,0)--(2,3,0)--(2,3,1)--cycle;
\draw  (2,3,1)--(2,3,0)--(3,3,0)--(3,3,1)--cycle;
\draw  (3,3,1)--(3,3,0)--(4,3,0)--(4,3,1)--cycle;
\end{tikzpicture}
\caption{Example of a polycube.} \label{fig:M42}
\end{figure}
\section{Preliminaries}
\label{Prel}
Let $(0,\vec{i},\vec{j})$ be an orthonormal coordinate.  
The area of a polyomino is the number of its cells, its width is the number of its columns and its height is the number of its lines.\\
A polyomino is column-convex if its intersection with any vertical line is connected.\\
A North (resp. East) step is a movement of one unit in $\vec{i}$-direction (resp. $\vec{j}$-direction).
A polyomino is directed if from a distinguished cell of the polyomino called root, we reach any other cell by a path that uses only North or East steps.   

Let $(0,\vec{i},\vec{j},\vec{k})$ be an orthonormal coordinate system. As for polyominoes, several parameters can be defined for a polycube. The \textit{volume} is the number of its cubes.
The \textit{width} (resp. \textit{height}, \textit{depth}) of a polycube is the difference between its greatest and its smallest indices according to $\vec{i}$ (resp. $\vec{j}$, $\vec{k}$). \\
A polycube is said to be \textit{directed} if each of its cells can be reached from a distinguished cell, called the \textit{root}, by a path only made of East (one unit in the $\vec{i}$-direction), North (one unit in $\vec{j}$-direction)  and Ahead (one unit in $\vec{k}$-direction) steps.
A \textit{stratum} is a polycube of width $1$. 
A \textit{plateau} is a rectangular stratum. A \textit{plateau polycube} is a polycube whose strata are plateaus \cite{ref1}. On can find more details on polyominoes in \cite{ref21} and in \cite{ref3} for polycubes.

\noindent
\begin{definition}
\label{d2_1}
The lateral area of a polycube is defined as the sum of its projections onto the planes $(\vec{i},\vec{j})$ and $(\vec{i},\vec{k})$. (Fig. \ref{M16})
\end{definition}

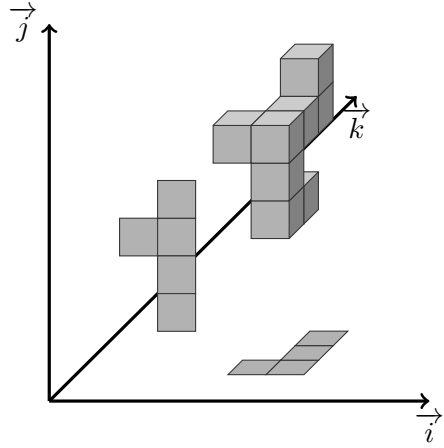
\begin{figure}[H]
\centering
\begin{tikzpicture}[scale=0.5]
\draw[very thick, ->] (2,0 ,13) -- (2,10,13) node[left]{$\overrightarrow{j}$};
\draw[very thick, ->] (2,0, 13) -- (2,0,-8) node[below]{$\overrightarrow{k}$};
\draw[very thick, ->] (2, 0,13) -- (12, 0,13) node[below]{$\overrightarrow{i}$};
	\foreach \x/\y/\z in {5/4/5,5/5/5,5/2/6,5/4/6,5/2/7,5/3/7,4/4/7,5/4/7}{
					\cube{\x}{\y}{\z};
				}
			\foreach \x/\y/\z in {0/0/3,1/0/3,1/1/3,1/-1/3,1/-2/3}{
					\freeza{\x}{\y}{\z};
				}		
				\foreach \x/\y/\z in {0/-7/-8,0/-7/-9,0/-7/-10,-1/-7/-8}{
					\buu{\x}{\y}{\z};
				}
					
\end{tikzpicture}
\caption{Example of polycube of width $2$ and lateral area $9$.}
\label{M16}
\end{figure}
Two polycubes with the same volume do not necessarily have the same lateral area (see Fig. \ref{M39}).
\begin{figure}[H]
\centering
\begin{subfigure}{.20\linewidth}
\begin{tikzpicture}[scale=.4,rotate=0]
			\begin{scope}[xshift = 0.4cm]
				
	\foreach \x/\y/\z in {5/2/6,5/2/7,5/3/7,5/4/7}{
					\cube{\x}{\y}{\z};
				}
			\foreach \x/\y/\z in {1/0/3,1/1/3,1/-1/3}{
					\freeza{\x}{\y}{\z};
				}		
				\foreach \x/\y/\z in {0/-7/-8,0/-7/-9}{
					\buu{\x}{\y}{\z};
				}
				\end{scope}	
				
			\end{tikzpicture}
\end{subfigure} \hspace{2.5cm}
\begin{subfigure}{.20\linewidth}
\begin{tikzpicture}[scale=.4,rotate=0]
			\begin{scope}[xshift = 1cm]
					\foreach \x/\y/\z in {5/2/7,5/3/7,5/4/7,6/2/7}{
					\cube{\x}{\y}{\z};
				}
			\foreach \x/\y/\z in {1/0/3,1/1/3,1/-1/3,2/-1/3}{
					\freeza{\x}{\y}{\z};
				}		
				\foreach \x/\y/\z in {0/-7/-8,1/-7/-8}{
					\buu{\x}{\y}{\z};
				}
				\end{scope}	
				
			\end{tikzpicture}
\end{subfigure}

\caption{Polycubes with the same volume and different lateral areas. }
\label{M39}
\end{figure}
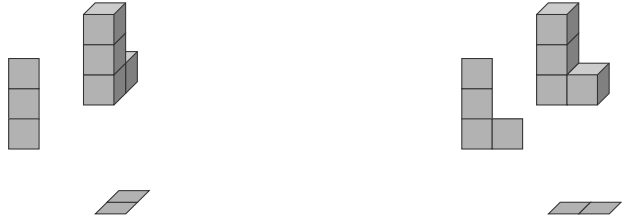
\noindent
The projection of a polycube onto a plane gives a polyomino, because if it is not the case, it implies that the projected object is not face-connected.\\
\noindent
The polyominoes obtained by the projections of a polycube onto the plane $(\vec{i},\vec{j})$ and $(\vec{i},\vec{k})$ have the same width as the polycube. \\ 
\noindent
Thus, if we project onto the planes $(\vec{i},\vec{j})$ and $(\vec{j},\vec{k})$ (resp. $(\vec{i},\vec{k})$ and $(\vec{j},\vec{k})$), we obtain a pair of polyominoes  not having the same characteristics. In this case, the two polyominoes will have the same height (resp. the height of the first one will be equal to the length of the second one).   

Let $\binom{n}{k}$ be the binomial coefficient. It counts the number of ways of choosing a subset of $k$ elements from a set of $n$ elements. The explicit formula is given by
\[\binom{n}{k} = \left\{ 
\begin{array}{l l}
  \frac{n!}{k!(n-k)!}, & \quad \text{for $0\leq k\leq n,$}\\
  0, & \quad \text{otherwise.}\\ \end{array} \right. \]

In the two-dimensional lattice, let us extend the notion of step defined for polyominoes. Let us introduce a new step, the North-East step. It is a movement of one unit in $\vec{i}$-direction and in $\vec{j}$-direction.
The total number $D(n,m)$ of paths from $(0,0)$ to $(n,m)$ only made of North, East and North-East steps is called Delannoy Number and is equal to
$\sum_{k=0}^m \binom{m}{k}\binom{n+m-k}{m}$.
\\
These numbers are also be computed using the following recursion
$$D(n,m)=D(n-1,m)+D(n,m-1)+D(n-1,m-1).$$

The first values of $D(n,m)$ are given in Fig. \ref{TrD}. 
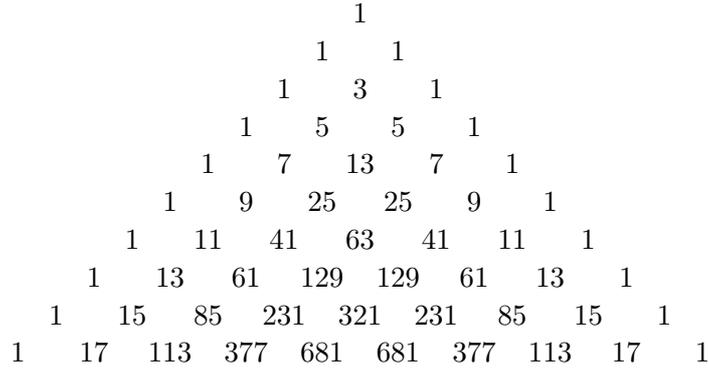
\begin{figure}[H]
\centering
\begin{tikzpicture}

\node at (0,0){1};
\node at (1,0){17};
\node at (2,0){113};
\node at (3,0){377};
\node at (4,0){681};
\node at (5,0){681};
\node at (6,0){377};
\node at (7,0){113};
\node at (8,0){17};
\node at (9,0){1};

\node at(0.5,0.5){1};
\node at (1.5,0.5){15};
\node at (2.5,0.5){85};
\node at(3.5,0.5){231};
\node at(4.5,0.5){321};
\node at(5.5,0.5){231};
\node at(6.5,0.5){85};
\node at(7.5,0.5){15};
\node at(8.5,0.5){1};

\node at(1,1){1};
\node at(2,1){13};
\node at(3,1){61};
\node at(4,1){129};
\node at(5,1){129};
\node at(6,1){61};
\node at(7,1){13};
\node at(8,1){1};

\node at(1.5,1.5){1};
\node at(2.5,1.5){11};
\node at(3.5,1.5){41};
\node at(4.5,1.5){63};
\node at(5.5,1.5){41};
\node at(6.5,1.5){11};
\node at(7.5,1.5){1};

\node at(2,2){1};
\node at(3,2){9};
\node at(4,2){25};
\node at(5,2){25};
\node at(6,2){9};
\node at(7,2){1};

\node at(2.5,2.5){1};
\node at(3.5,2.5){7};
\node at(4.5,2.5){13};
\node at(5.5,2.5){7};
\node at(6.5,2.5){1};

\node at(3,3){1};
\node at(4,3){5};
\node at(5,3){5};
\node at(6,3){1};

\node at(3.5,3.5){1};
\node at(4.5,3.5){3};
\node at(5.5,3.5){1};

\node at(4,4){1};
\node at(5,4){1};

\node at(4.5,4.5){1};
\end{tikzpicture}
\caption{Tribonacci Triangle.}
\label{TrD}
\end{figure}

For more details about Delannoy Numbers, one can refer to \cite{ref20}. 
\section{Enumeration of directed plateau polycubes}
\label{EDP}
In order to enumerate plateau polycubes, we first characterize their projections onto the planes $(\vec{i},\vec{j})$ and $(\vec{i},\vec{k})$.\\
In \cite{ref2} (Section $5$), it is mentioned that the projection of a directed plateau polycube gives a directed column-convex polyominoes on each plane. Moreover, from any two directed column-convex polyominoes in these planes corresponds a unique directed plateau polycube (see Fig. \ref{pl}).

\begin{figure}[H]
\centering
\begin{tikzpicture}[scale=0.45]
\fill[color=gray!60](2,3,4)--(2,3,3)--(2,4,3)--(2,4,4)--cycle;
\fill[color=gray!60](2,3,3)--(2,3,2)--(2,4,2)--(2,4,3)--cycle;
\fill[color=gray!60](2,3,2)--(2,3,1)--(2,4,1)--(2,4,2)--cycle;

\draw(2,3,2)--(2,3,1)--(2,4,1)--(2,4,2)--cycle;

\fill[color=gray!60](3,4,4)--(3,4,3)--(3,5,3)--(3,5,4)--cycle;
\fill[color=gray!60](3,4,3)--(3,4,2)--(3,5,2)--(3,5,3)--cycle;
\fill[color=gray!60](3,5,4)--(3,5,3)--(3,6,3)--(3,6,4)--cycle;
\fill[color=gray!60](3,5,3)--(3,5,2)--(3,6,2)--(3,6,3)--cycle;
\fill[color=gray!60](2,6,4)--(2,6,3)--(2,7,3)--(2,7,4)--cycle;
\fill[color=gray!60](2,6,3)--(2,6,2)--(2,7,2)--(2,7,3)--cycle;
\fill[color=gray!60](2,6,2)--(2,6,1)--(2,7,1)--(2,7,2)--cycle;
\fill[color=gray!60](2,5,2)--(2,5,1)--(2,6,1)--(2,6,2)--cycle;

\draw(2,3,4)--(2,3,3)--(2,4,3)--(2,4,4)--cycle;
\draw(2,3,3)--(2,3,2)--(2,4,2)--(2,4,3)--cycle;

\draw(3,4,4)--(3,4,3)--(3,5,3)--(3,5,4)--cycle;
\draw(3,4,3)--(3,4,2)--(3,5,2)--(3,5,3)--cycle;
\draw(3,5,4)--(3,5,3)--(3,6,3)--(3,6,4)--cycle;
\draw(3,5,3)--(3,5,2)--(3,6,2)--(3,6,3)--cycle;
\draw(2,6,4)--(2,6,3)--(2,7,3)--(2,7,4)--cycle;
\draw(2,6,3)--(2,6,2)--(2,7,2)--(2,7,3)--cycle;
\draw(2,6,2)--(2,6,1)--(2,7,1)--(2,7,2)--cycle;
\draw(2,5,2)--(2,5,1)--(2,6,1)--(2,6,2)--cycle;

\fill[color=gray!20](0,5,4)--(1,5,4)--(1,5,3)--(0,5,3)--cycle;
\fill[color=gray!20](1,7,4)--(2,7,4)--(2,7,3)--(1,7,3)--cycle;
\fill[color=gray!20](1,7,3)--(2,7,3)--(2,7,2)--(1,7,2)--cycle;
\fill[color=gray!20](1,7,2)--(2,7,2)--(2,7,1)--(1,7,1)--cycle;
\fill[color=gray!20](2,6,4)--(3,6,4)--(3,6,3)--(2,6,3)--cycle;
\fill[color=gray!20](2,6,3)--(3,6,3)--(3,6,2)--(2,6,2)--cycle;

\draw(0,5,4)--(1,5,4)--(1,5,3)--(0,5,3)--cycle;
\draw(1,7,4)--(2,7,4)--(2,7,3)--(1,7,3)--cycle;
\draw(1,7,3)--(2,7,3)--(2,7,2)--(1,7,2)--cycle;
\draw(1,7,2)--(2,7,2)--(2,7,1)--(1,7,1)--cycle;
\draw(2,6,4)--(3,6,4)--(3,6,3)--(2,6,3)--cycle;
\draw(2,6,3)--(3,6,3)--(3,6,2)--(2,6,2)--cycle;

\fill[color=gray!40](0,3,13)--(1,3,13)--(1,4,13)--(0,4,13)--cycle;
\fill[color=gray!40](1,3,13)--(2,3,13)--(2,4,13)--(1,4,13)--cycle; 
\fill[color=gray!40](0,4,13)--(1,4,13)--(1,5,13)--(0,5,13)--cycle;
\fill[color=gray!40](1,4,13)--(2,4,13)--(2,5,13)--(1,5,13)--cycle; 
\fill[color=gray!40](1,5,13)--(2,5,13)--(2,6,13)--(1,6,13)--cycle; 
\fill[color=gray!40](2,4,13)--(3,4,13)--(3,5,13)--(2,5,13)--cycle; 
\fill[color=gray!40](2,5,13)--(3,5,13)--(3,6,13)--(2,6,13)--cycle;
\fill[color=gray!40](1,6,13)--(2,6,13)--(2,7,13)--(1,7,13)--cycle;

\fill[color=gray!40](0,0,4)--(1,0,4)--(1,0,3)--(0,0,3)--cycle;
\fill[color=gray!40](1,0,4)--(2,0,4)--(2,0,3)--(1,0,3)--cycle;
\fill[color=gray!40](1,0,3)--(2,0,3)--(2,0,2)--(1,0,2)--cycle;
\fill[color=gray!40](1,0,2)--(2,0,2)--(2,0,1)--(1,0,1)--cycle;
\fill[color=gray!40](2,0,4)--(3,0,4)--(3,0,3)--(2,0,3)--cycle;
\fill[color=gray!40](2,0,3)--(3,0,3)--(3,0,2)--(2,0,2)--cycle;

\fill[color=gray!40](0,3,4)--(1,3,4)--(1,4,4)--(0,4,4)--cycle;
\fill[color=gray!40](1,3,4)--(2,3,4)--(2,4,4)--(1,4,4)--cycle; 
\fill[color=gray!40](0,4,4)--(1,4,4)--(1,5,4)--(0,5,4)--cycle;
\fill[color=gray!40](1,4,4)--(2,4,4)--(2,5,4)--(1,5,4)--cycle; 
\fill[color=gray!40](1,5,4)--(2,5,4)--(2,6,4)--(1,6,4)--cycle; 
\fill[color=gray!40](2,4,4)--(3,4,4)--(3,5,4)--(2,5,4)--cycle; 
\fill[color=gray!40](2,5,4)--(3,5,4)--(3,6,4)--(2,6,4)--cycle;
\fill[color=gray!40](1,6,4)--(2,6,4)--(2,7,4)--(1,7,4)--cycle;

\draw (0,3,13)--(1,3,13)--(1,4,13)--(0,4,13)--cycle;
\draw(1,3,13)--(2,3,13)--(2,4,13)--(1,4,13)--cycle; 
\draw(0,4,13)--(1,4,13)--(1,5,13)--(0,5,13)--cycle;
\draw(1,4,13)--(2,4,13)--(2,5,13)--(1,5,13)--cycle; 
\draw(1,5,13)--(2,5,13)--(2,6,13)--(1,6,13)--cycle; 
\draw(2,4,13)--(3,4,13)--(3,5,13)--(2,5,13)--cycle; 
\draw(2,5,13)--(3,5,13)--(3,6,13)--(2,6,13)--cycle;
\draw(1,6,13)--(2,6,13)--(2,7,13)--(1,7,13)--cycle;

\draw(0,0,4)--(1,0,4)--(1,0,3)--(0,0,3)--cycle;
\draw(1,0,4)--(2,0,4)--(2,0,3)--(1,0,3)--cycle;
\draw(1,0,3)--(2,0,3)--(2,0,2)--(1,0,2)--cycle;
\draw(1,0,2)--(2,0,2)--(2,0,1)--(1,0,1)--cycle;
\draw(2,0,4)--(3,0,4)--(3,0,3)--(2,0,3)--cycle;
\draw(2,0,3)--(3,0,3)--(3,0,2)--(2,0,2)--cycle;

\draw(0,3,4)--(1,3,4)--(1,4,4)--(0,4,4)--cycle;
\draw(1,3,4)--(2,3,4)--(2,4,4)--(1,4,4)--cycle; 
\draw(0,4,4)--(1,4,4)--(1,5,4)--(0,5,4)--cycle;
\draw(1,4,4)--(2,4,4)--(2,5,4)--(1,5,4)--cycle; 
\draw(1,5,4)--(2,5,4)--(2,6,4)--(1,6,4)--cycle; 
\draw(2,4,4)--(3,4,4)--(3,5,4)--(2,5,4)--cycle; 
\draw(2,5,4)--(3,5,4)--(3,6,4)--(2,6,4)--cycle;
\draw(1,6,4)--(2,6,4)--(2,7,4)--(1,7,4)--cycle;
\end{tikzpicture}
\caption{Directed plateau polycube of width 3 and lateral area 14.}
\label{pl}
\end{figure}
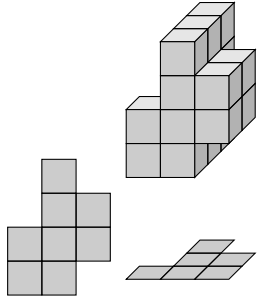

The number of directed column-convex polyominoes is known and is given in the following lemma,
\begin{lemme}
\label{l3_1}{\normalfont{(\cite{ref9})}}
For $1\leq k \leq n$, the number of directed column-convex polyominoes having $k$ columns and an area $n$ is $$\binom{n+k-2}{n-k}.$$
\end{lemme}
From Lemma \ref{l3_1} we deduce the following theorem.

\begin{theorem}
\label{t3_1}
Let $s_{k,n}$ denote the number of directed plateau polycubes of width $k$ and lateral area $n$. Then, for $k\geq 1$ and $n\geq 2k$, we have 
$$s_{k,n}=\sum_{i=k}^{n-k}\binom{i+k-2}{i-k}\binom{n-i+k-2}{n-i-k}.$$
\end{theorem}

\begin{proof} 
By definition, the lateral area of a directed plateau polycube is the sum of the areas of the two directed column-convex polyominoes obtained from its projection. So, if a polycube has a width $k$ and a lateral area $n$, then both polyominoes obtained have $k$ columns. If the area of one of them is $i$ then the area of the other is $n-i$. 
Lemma \ref{l3_1} gives that the number of directed column-convex polyominoes having an area $i$ and $k$ columns is equal to $\binom{i+k-2}{i-k}$. Therefore, the number of plateau polycubes having a lateral area $n$ and width $k$ and whose projections on $(\vec{i},\vec{j})$ give a polyomino of area $i$ and the projections on $(\vec{i},\vec{k})$ give a polyomino of area $n-i$ is equal to $\binom{i+k-2}{i-k}\binom{n-i+k-2}{n-i-k}$. By summing all possible values of $i$, we obtain the formula.
\end{proof}

\noindent
From Theorem \ref{t3_1} and binomial coefficient properties, we deduce the following explicit expression.
\begin{theorem}
\label{t3_2} 
Let $k\geq 1$ and $n\geq 2k$,
$$s_{k,n}=\binom{n+2k-3}{n-2k}.$$
\end{theorem}
\begin{proof}
From Theorem \ref{t3_1}, we have
\begin{eqnarray*}
s_{k,n} &=&\sum_{i=k}^{n-k}\binom{i+k-2}{i-k}\binom{n-i+k-2}{n-i-k}. \nonumber
\end{eqnarray*}
\noindent
Setting $j=i-k$, $m=n-2k$ and $a=2k-2$, we obtain\\
\noindent
\begin{eqnarray*}
s_{k,n}=\sum_{j=0}^{m}\binom{j+a}{j}\binom{a+m-j}{m-j}.\nonumber
\end{eqnarray*}
\noindent
A variant of Vandermonde convolution \cite{ref23} is
$$\sum_{j=0}^{m}\binom{j+a}{j}\binom{a+m-j}{m-j}=\binom{2a+m+1}{m}.$$
\noindent
By replacing $a$, $j$ and $m$ by their values, we get the formula.
\end{proof}
\noindent
\\
From this result, we can deduce a formula for $S_{k}(t)$, the \textit{generating function} for the directed plateau polycubes of width $k$ according to the lateral area where $$S_{k}(t)=\sum_{n\geq 0}s_{k,n}t^n.$$ 
So,
\begin{align*}
S_{k}(t)&=\sum_{n\geq 0}\sum_{i=k}^{n-k}\binom{i+k-2}{i-k}\binom{n-i+k-2}{n-i-k}t^n. 
\end{align*}
By convention, $S_0(t)=0$.\\
For $i \leq k$ and $i \geq n-k$, $s_{k,n}=0$. Therefore
\begin{align*}
S_{k}(t)&=\sum_{n\geq 0}\sum_{i=0}^{n}\binom{i+k-2}{i-k}\binom{n-i+k-2}{n-i-k}t^n \\ \\
&=\bigg(\sum_{m\geq 0}\binom{m+k-2}{m-k}t^m\bigg)\bigg(\sum_{l\geq 0}\binom{l+k-2}{l-k}t^l\bigg).
\end{align*}

\noindent
It is proved by, Barcucci et \textit{al.}, \cite{ref222} that the generating function of directed column-convex polyominoes having exactly $k$ columns is 
$$\sum_{n\geq 0}\binom{n+k-2}{n-k}t^n=\frac{t^k}{(1-t)^{2k-1}}.$$ 
Therefore,
\begin{proposition}
\label{p3_1} For $k\geq 1$, we have
\begin{align*}
S_{k}(t)&=\frac{t^{2k}}{(1-t)^{4k-2}}.\\ 
\end{align*}
\end{proposition}
\begin{proof}
Let $$S(x,t):=\sum_{k\geq 1}S_{k}(t)x^k,$$ be the generating function of directed plateau polycubes according to the lateral area (coded by $t$) and the width (coded by $x$) then,
$$S(x,t)=\frac{xt^2(1-t)^2}{(1-t)^4-xt^2}.$$
For $x=1$ in $S(x,t)$ we obtain the following result.
\end{proof}
\begin{theorem}
\label{t3_3}
An expression of, $S(t)$, the generating function of directed plateau polycubes according to the lateral area is
$$S(t)=\frac{t^2(1-t)^2}{(1-t)^4-t^2}.$$
\end{theorem}

\section{Enumeration of plateau polycubes}
\label{EP}
As for the directed case, we have at first to characterize their projections.\\
It is known, from \cite{ref2}, that a plateau polycube can be obtained from two column-convex polyominoes and their number of columns is equal to the width of the polycube, (see Fig. \ref{pl}).\\

An immediate consequence is Theorem \ref{t4_1}. To prove it, we use the following lemma,
\begin{lemme}{\normalfont(\cite{ref17})}
\label{l4_1}
For $1\leq k \leq n$, the number of column-convex polyominoes having $k$ columns and an area $n$ is
 $$\sum_{i,j\geq 0}\binom{k-i-1}{i}\binom{2k-j-2}{j}\binom{k-2i-1}{n-k-i-j}.$$
\end{lemme}

\begin{theorem}
\label{t4_1}
Let $r_{k,m}$ denote the number of plateau polycubes of width $k$ and lateral area $m$. Then for $k\geq 1$ and $m\geq 2k$, we have following the convolution formula,
$$r_{k,m}=\sum_{i=k}^{m-k}\alpha_k(i)\alpha_k(m-i),$$

\noindent where \setlength\abovedisplayskip{0pt}
$$\alpha_{k}(u)=\sum_{i,j\geq 0}\binom{k-i-1}{i}\binom{2k-j-2}{j}\binom{k-2i-1}{u-k-i-j}.$$ 
\end{theorem}
\begin{proof}
The proof is similar to the one in Theorem \ref{t3_1}.
\end{proof}
\noindent
\\
From Theorem \ref{t4_1}, we deduce the generating function for a fixed width $k$.\\

\begin{theorem}
\label{t4_2}
Let $R_k(p)$ be the generating function of plateau polycubes of width $k$ according to the lateral area. Then,
$$R_k(p)=C_{k-1}(p)^2,$$
where $$C_k(p)=\frac{p^{k+1}}{(1-p)^{2k+1}}\sum_{i=0}^{k}D_{k-i,i}p^i.$$
\end{theorem}

\begin{proof}
The proof is the same as in Theorem \ref{t3_1}, the generating function of polyominoes with $k$ columns according to the area is \cite{ref2} 
$$\frac{p^k}{(1-p)^{2k-1}}\sum_{i=0}^{k-1}D_{k-i-1,i}p^i.$$

\end{proof}
\noindent
\\
Notice that, $C_k$ is the generating function of the Anti-Diagonal sequences lyning in the Tribonacci Triangle at level $k$; times $\frac{p^{k+1}}{(1-p)^{2k+1}}$ (see Fig. \ref{ADDT}).

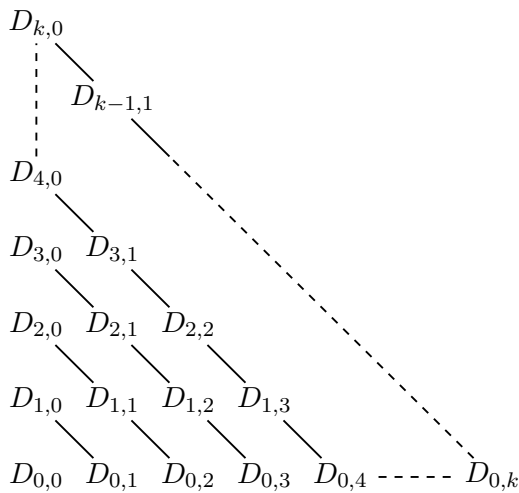
\begin{figure}[H]
\centering
\begin{tikzpicture}

\draw[line width=0.25mm](0.75,0.25)--(0.25,0.75);
\draw[line width=0.25mm](0.25,1.75)--(0.75,1.25);
\draw[line width=0.25mm](1.25,0.75)--(1.75,0.25);
\draw[line width=0.25mm](2.75,0.25)--(2.25,0.75);
\draw[line width=0.25mm](1.75,1.25)--(1.25,1.75);
\draw[line width=0.25mm](0.75,2.25)--(0.25,2.75);
\draw[line width=0.25mm](0.25,3.75)--(0.75,3.25);
\draw[line width=0.25mm](1.25,2.75)--(1.75,2.25);
\draw[line width=0.25mm](2.25,1.75)--(2.75,1.25);
\draw[line width=0.25mm](3.25,0.75)--(3.75,0.25);
\draw[line width=0.25mm,dashed](4.5,0)--(5.5,0);
\draw[line width=0.25mm,dashed](0,4.25)--(0,5.75);
\draw[line width=0.25mm](0.25,5.75)--(0.75,5.25);
\draw[line width=0.25mm,dashed](5.75,0.25)--(1.75,4.25);
\draw[line width=0.25mm](1.25,4.75)--(1.75,4.25);
\node at (0,0){$D_{0,0}$};
\node at (1,0){$D_{0,1}$};
\node at (2,0){$D_{0,2}$};
\node at (3,0){$D_{0,3}$};
\node at (4,0){$D_{0,4}$};
\node at (6,0){$D_{0,k}$};

\node at (0,1){$D_{1,0}$};
\node at (1,1){$D_{1,1}$};
\node at (2,1){$D_{1,2}$};
\node at (3,1){$D_{1,3}$};

\node at (0,2){$D_{2,0}$};
\node at (1,2){$D_{2,1}$};
\node at (2,2){$D_{2,2}$};

\node at (0,3){$D_{3,0}$};
\node at (1,3){$D_{3,1}$};

\node at (0,4){$D_{4,0}$};

\node at (0,6){$D_{k,0}$};
\node at (1,5){$D_{k-1,1}$};

\end{tikzpicture}
\caption{Anti-Diagonal transversals of the Tribonacci Triangle.}
\label{ADDT}
\end{figure}

\noindent

The first values of $r_{k,m}$ are given in Table \ref{t1}. From these values, we establish some asymptotic results.

\begin{sidewaystable}
\centering
\begin{tabular}{|P{1cm}|P{2cm}P{2cm}P{2cm}P{2cm}P{2cm}P{2cm}P{2cm}|}
\hline
$m \backslash k $  & 1& 2& 3&4&5&6&7\\ \hline
2 &\multicolumn{1}{P{2cm}:}{1} &0&0&0&0&0&0 \\
3&\multicolumn{1}{P{2cm}:}{2} &0&0&0&0&0&0\\ \cdashline{3-3}
4&3&\multicolumn{1}{P{2cm}:}{1} &0&0&0&0&0\\
5&4&\multicolumn{1}{P{2cm}:}{8} &0&0&0&0&0\\ \cdashline{4-4}
6&5&34&\multicolumn{1}{P{2cm}:}{1} &0&0&0&0\\
7&6&104&\multicolumn{1}{P{2cm}:}{16} &0&0&0&0\\\cdashline{5-5}
8&7&259&126&\multicolumn{1}{P{2cm}:}{1} &0&0&0\\
9&8&560&666&\multicolumn{1}{P{2cm}:}{24} &0&0&0\\\cdashline{6-6}
10&9&1092&2701&280&\multicolumn{1}{P{2cm}:}{1} &0&0\\
11&10&1968&9052&2152&\multicolumn{1}{P{2cm}:}{32} &0&0\\\cdashline{7-7}
12&11&3333&26257&12418&498&\multicolumn{1}{P{2cm}:}{1} &0\\
13&12&5368&68002&57922&5080&\multicolumn{1}{P{2cm}:}{40} &0\\\cdashline{8-8}
14&13&8294&160732&229048&38567&780&1\\
15&14&12376&352352&793144&234178&9960&48 \\
16&15&17927&725153&2462851&1191540&94318&1126\\
17&16&25312&1414348&6980624&5249012&710584&17304\\
16&17&34952&2633878&18309136&20506003&4457930&196953\\
17&18&47328&4711448&44921072&72354830&24048920&1778848\\
18&19&62985&8135078&103994372&233915707&114248221&13331808\\
19&20&82536&13613804&228782192&7008599688&486806272&85565538\\
20&21&106666&22155539&48109488&1964393375&1887595700&481457252\\
21&22&136136&35165504&971764880&5190268342&6738878720&2418499500\\
22&23&171787&54569064&1893273221&13010791823&22364636385&11003497968\\
23&24&214544&82963254&3570426344&31111765764&69550800504&45877909970\\ \hline

\end{tabular}

\caption{The first values of $r_{k,m}$.}
\label{t1}
\end{sidewaystable}

\section{Asymptotic Results}
\label{AR}
\noindent
In this section, we present results obtained combinatorially. They also can be obtained in algebraic way from Lemma \ref{l4_1} and Theorem \ref{t4_1}. 
\subsection{Column-convex polyominoes}
\noindent
Let $h_{k,n}$ be the number of column-convex polyominoes having $k$ columns and area $n$. The first values are given in Table \ref{t2}.
\begin{table}[H]
\begin{tabular}{|c|P{1cm}P{1cm}P{1cm}P{1cm}P{1cm}P{1cm}P{1cm}P{1cm}P{1cm}P{1cm}|}\hline
$n \backslash k$  & 1 & 2& 3&4&5&6&7&8&9&10 \\ \hline
1&1&0&0&0&0&0&0&0&0&0\\
2&1&1&0&0&0&0&0&0&0&0\\
3&1&4&1&0&0&0&0&0&0&0\\
4&1&9&8&1&0&0&0&0&0&0\\
5&1&16&31&12&1&0&0&0&0&0\\
6&1&25&85&68&16&1&0&0&0&0\\
7&1&36&190&260&121&20&1&0&0&0\\
8&1&49&371&777&604&190&24&1&0&0\\
9&1&64&658&1960&2299&1180&275&28&1&0\\
10&1&81&1086&4368&7221&5509&2052&376&32&1\\ \hline
\end{tabular}
\caption{The first values of $h_{k,n}$.}
\label{t2}

\end{table}

Here are given some specific values,
\begin{theorem}
\label{t5_1}
\mbox{}
\begin{enumerate}
\item $h_{k,k}=1$ for $k\geq 1$.
\item $h_{k,k+1}=4k-4$ for $k\geq2$.
\item $h_{k,k+2}=8k^2-19k+16$ for $k\geq 3$.
\end{enumerate}
\end{theorem}
\begin{proof}
The proofs of these formulas are based on a unique principle. We start from the polyomino having area $k$ and $k$ columns. Next, we build all polyominoes having of area $k+i$ and $k$ columns by adding $i$ cells, cutting and gluing it. Thus, we only detail the proof of the value $h_{k,k+2}$.\\
Let us consider the polyomino of area $k$ and $k$ columns and let us enumerate all the ways to inject two cells to obtain column-convex polyomino. There are two main cases: 
\begin{itemize}
\item The two cells are inserted on the same column:
		\begin{itemize}
			\item When we insert cells in the first column, we have $3$ possibilities (see Fig. \ref{a1}).
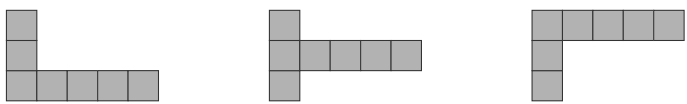
\begin{figure}[H]
\centering
\begin{subfigure}{.20\linewidth}
\begin{tikzpicture}[scale=.4,rotate=0]
			\begin{scope}[xshift = 1cm]
					\foreach \x/\y in 
					{0/0,0/1,0/2,1/0,2/0,3/0,4/0}
					{
						\cell{\x}{\y}
					}
				\end{scope}	
				
			\end{tikzpicture}
\end{subfigure} \hspace{0.000005cm}
\begin{subfigure}{.20\linewidth}
\begin{tikzpicture}[scale=.4,rotate=0]
			\begin{scope}[xshift = 1cm]
					\foreach \x/\y in 
					{0/0,0/1,0/2,1/1,2/1,3/1,4/1}
					{
						\cell{\x}{\y}
					}
				\end{scope}	
				
			\end{tikzpicture}
\end{subfigure} \hspace{0.000005cm}
\begin{subfigure}{.20\linewidth}
\begin{tikzpicture}[scale=.4,rotate=0]
			\begin{scope}[xshift = 1cm]
					\foreach \x/\y in 
					{0/0,0/1,0/2,1/2,2/2,3/2,4/2}
					{
						\cell{\x}{\y}
					}
				\end{scope}	
				
			\end{tikzpicture}
\end{subfigure}

\caption{Insertion of two cells in the first column.}
\label{a1}
\end{figure}

			\item  The insertion on the last column is similar to the previous subcase. We have $3$ possibilities.
			\item  In the last subcase there are $9(k-2)$ possibilities: we add a cell on $k-2$ different columns. And in an adding, we have $3$ insertions on the left and for each of them we have $3$ insertions on the right (see Fig. \ref{a2}).
			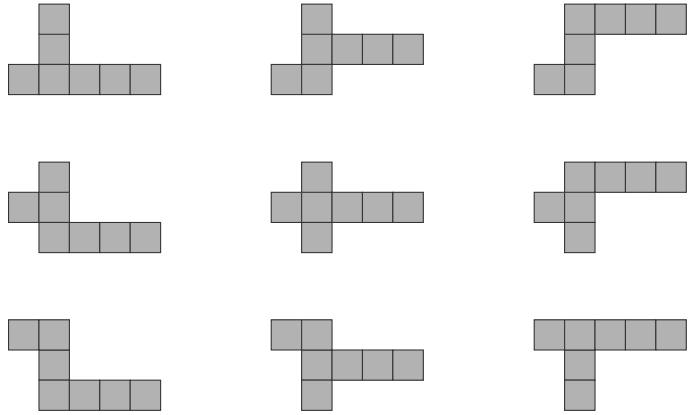
\begin{figure}[H]
\centering
  
\begin{subfigure}{.20\linewidth}
\begin{tikzpicture}[scale=.4,rotate=0]
			\begin{scope}[xshift = 6cm]
					\foreach \x/\y in 
					{0/0,1/1,1/2,1/0,2/0,3/0,4/0}
					{
						\cell{\x}{\y}
					}
				\end{scope}	
				
			\end{tikzpicture}
\end{subfigure} \hspace{0.000005cm}
\begin{subfigure}{.20\linewidth}
\begin{tikzpicture}[scale=.4,rotate=0]
			\begin{scope}[xshift = 6cm]
					\foreach \x/\y in 
					{0/0,1/1,1/2,1/0,2/1,3/1,4/1}
					{
						\cell{\x}{\y}
					}
				\end{scope}	
				
			\end{tikzpicture}
\end{subfigure} \hspace{0.000005cm}
\begin{subfigure}{.20\linewidth}
\begin{tikzpicture}[scale=.4,rotate=0]
			\begin{scope}[xshift = 6cm]
					\foreach \x/\y in 
					{0/0,1/1,1/2,1/0,2/2,3/2,4/2}
					{
						\cell{\x}{\y}
					}
				\end{scope}	
				
			\end{tikzpicture}
\end{subfigure}

\par\bigskip
\par\bigskip
\begin{subfigure}{.20\linewidth}
\begin{tikzpicture}[scale=.4,rotate=0]
			\begin{scope}[xshift = 6cm]
					\foreach \x/\y in 
					{0/1,1/1,1/2,1/0,2/0,3/0,4/0}
					{
						\cell{\x}{\y}
					}
				\end{scope}	
				
			\end{tikzpicture}
\end{subfigure} \hspace{0.000005cm}
\begin{subfigure}{.20\linewidth}
\begin{tikzpicture}[scale=.4,rotate=0]
			\begin{scope}[xshift = 6cm]
					\foreach \x/\y in 
					{0/1,1/1,1/2,1/0,2/1,3/1,4/1}
					{
						\cell{\x}{\y}
					}
				\end{scope}	
				
			\end{tikzpicture}
\end{subfigure} \hspace{0.000005cm}
\begin{subfigure}{.20\linewidth}
\begin{tikzpicture}[scale=.4,rotate=0]
			\begin{scope}[xshift = 6cm]
					\foreach \x/\y in 
					{0/1,1/1,1/2,1/0,2/2,3/2,4/2}
					{
						\cell{\x}{\y}
					}
				\end{scope}	
				
			\end{tikzpicture}
\end{subfigure}
\par\bigskip
\par\bigskip
\begin{subfigure}{.20\linewidth}
\begin{tikzpicture}[scale=.4,rotate=0]
			\begin{scope}[xshift = 6cm]
					\foreach \x/\y in 
					{0/2,1/1,1/2,1/0,2/0,3/0,4/0}
					{
						\cell{\x}{\y}
					}
				\end{scope}	
				
			\end{tikzpicture}
\end{subfigure} \hspace{0.000005cm} 
\begin{subfigure}{.20\linewidth}
\begin{tikzpicture}[scale=.4,rotate=0]
			\begin{scope}[xshift = 6cm]
					\foreach \x/\y in 
					{0/2,1/1,1/2,1/0,2/1,3/1,4/1}
					{
						\cell{\x}{\y}
					}
				\end{scope}	
				
			\end{tikzpicture}
\end{subfigure} \hspace{0.000005cm}
\begin{subfigure}{.20\linewidth}
\begin{tikzpicture}[scale=.4,rotate=0]
			\begin{scope}[xshift = 6cm]
					\foreach \x/\y in 
					{0/2,1/1,1/2,1/0,2/2,3/2,4/2}
					{
						\cell{\x}{\y}
					}
				\end{scope}	
				
			\end{tikzpicture}
\end{subfigure}
  \caption{Insertion of two cells in a middle column.}
  \label{a2}
\end{figure}

		\end{itemize}
\item The two cells are inserted on two different columns:		
		\begin{itemize}
		\item We add two cells on two successive columns:
			\begin{itemize}
			\item[$\blacktriangleright$] We insert a cell in the first and in the second columns, there is $6$ possibilities (see Fig. \ref{a8}).
			\item [$\blacktriangleright$]One cell is inserted on the last column and the other one on the before last column. It is similar to the previous case. We have $6$ possibilities.
			\begin{figure}[H]
\centering
  
\begin{subfigure}{.20\linewidth}
\begin{tikzpicture}[scale=.4,rotate=0]
			\begin{scope}[xshift = 6cm]
					\foreach \x/\y in 
					{0/0,0/1,1/0,1/1,2/0,3/0,4/0}
					{
						\cell{\x}{\y}
					}
				\end{scope}	
				
			\end{tikzpicture}
\end{subfigure} \hspace{0.000005cm}
\begin{subfigure}{.20\linewidth}
\begin{tikzpicture}[scale=.4,rotate=0]
			\begin{scope}[xshift = 6cm]
					\foreach \x/\y in 
					{0/0,0/1,1/0,1/1,2/1,3/1,4/1}
					{
						\cell{\x}{\y}
					}
				\end{scope}	
				
			\end{tikzpicture}
\end{subfigure} \hspace{0.000005cm}
\begin{subfigure}{.20\linewidth}
\begin{tikzpicture}[scale=.4,rotate=0]
			\begin{scope}[xshift = 6cm]
					\foreach \x/\y in 
					{0/0,0/1,1/1,1/2,2/1,3/1,4/1}
					{
						\cell{\x}{\y}
					}
				\end{scope}	
				
			\end{tikzpicture}
\end{subfigure}

\par\bigskip
\par\bigskip
\begin{subfigure}{.20\linewidth}
\begin{tikzpicture}[scale=.4,rotate=0]
			\begin{scope}[xshift = 6cm]
					\foreach \x/\y in 
					{0/0,0/1,1/1,1/2,2/2,3/2,4/2}
					{
						\cell{\x}{\y}
					}
				\end{scope}	
				
			\end{tikzpicture}
\end{subfigure} \hspace{0.000005cm}
\begin{subfigure}{.20\linewidth}
\begin{tikzpicture}[scale=.4,rotate=0]
			\begin{scope}[xshift = 6cm]
					\foreach \x/\y in 
					{0/0,0/1,1/0,1/-1,2/-1,3/-1,4/-1}
					{
						\cell{\x}{\y}
					}
				\end{scope}	
				
			\end{tikzpicture}
\end{subfigure} \hspace{0.000005cm}
\begin{subfigure}{.20\linewidth}
\begin{tikzpicture}[scale=.4,rotate=0]
			\begin{scope}[xshift = 6cm]
					\foreach \x/\y in 
					{0/0,0/1,1/0,1/-1,2/0,3/0,4/0}
					{
						\cell{\x}{\y}
					}
				\end{scope}	
				
			\end{tikzpicture}
\end{subfigure}
\caption{Insertion of two cells in the first and the second columns.}
  \label{a8}
\end{figure}
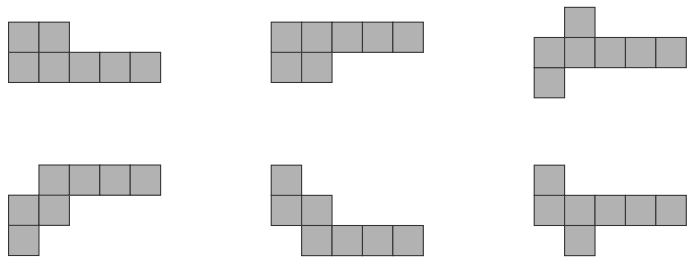
			\item[$\blacktriangleright$] The two cells are inserted in two different columns from the first and the last. We have in this subcase $12(k-3)$ possibilities (see Fig. \ref{a3}).
\begin{figure}[H]
\centering
  
\begin{subfigure}{.20\linewidth}
\begin{tikzpicture}[scale=.4,rotate=0]
			\begin{scope}[xshift = 1cm]
					\foreach \x/\y in 
					{0/0,1/0,1/1,2/0,2/1,3/0,4/0}
					{
						\cell{\x}{\y}
					}
				\end{scope}	
				
			\end{tikzpicture}
\end{subfigure} \hspace{0.000005cm}
\begin{subfigure}{.20\linewidth}
\begin{tikzpicture}[scale=.4,rotate=0]
			\begin{scope}[xshift = 1cm]
					\foreach \x/\y in 
				{0/0,1/0,1/1,2/0,2/1,3/1,4/1}
					{
						\cell{\x}{\y}
					}
				\end{scope}	
				
			\end{tikzpicture}
\end{subfigure} \hspace{0.000005cm}
\begin{subfigure}{.20\linewidth}
\begin{tikzpicture}[scale=.4,rotate=0]
			\begin{scope}[xshift = 1cm]
					\foreach \x/\y in 
					{0/1,1/0,1/1,2/0,2/1,3/0,4/0}
					{
						\cell{\x}{\y}
					}
				\end{scope}	
				
			\end{tikzpicture}
\end{subfigure}
\par\bigskip
\par\bigskip
\begin{subfigure}{.20\linewidth}
\begin{tikzpicture}[scale=.4,rotate=0]
			\begin{scope}[xshift = 1cm]
					\foreach \x/\y in 
					{0/1,1/0,1/1,2/0,2/1,3/1,4/1}
					{
						\cell{\x}{\y}
					}
				\end{scope}	
				
			\end{tikzpicture}
\end{subfigure}
\hspace{0.000005cm}
\begin{subfigure}{.20\linewidth}
\begin{tikzpicture}[scale=.4,rotate=0]
			\begin{scope}[xshift = 1cm]
					\foreach \x/\y in 
					{0/0,1/0,1/1,2/1,2/2,3/1,4/1}
					{
						\cell{\x}{\y}
					}
				\end{scope}	
				
			\end{tikzpicture}
\end{subfigure} \hspace{0.000005cm}
\begin{subfigure}{.20\linewidth}
\begin{tikzpicture}[scale=.4,rotate=0]
			\begin{scope}[xshift = 1cm]
					\foreach \x/\y in 
				{0/0,1/0,1/1,2/1,2/2,3/2,4/2}
					{
						\cell{\x}{\y}
					}
				\end{scope}	
				
			\end{tikzpicture}
\end{subfigure} 
\par\bigskip
\par\bigskip
\begin{subfigure}{.20\linewidth}
\begin{tikzpicture}[scale=.4,rotate=0]
			\begin{scope}[xshift = 1cm]
					\foreach \x/\y in 
					{0/1,1/0,1/1,2/1,2/2,3/1,4/1}
					{
						\cell{\x}{\y}
					}
				\end{scope}	
				
			\end{tikzpicture}
\end{subfigure}
\hspace{0.000005cm}
\begin{subfigure}{.20\linewidth}
\begin{tikzpicture}[scale=.4,rotate=0]
			\begin{scope}[xshift = 1cm]
					\foreach \x/\y in 
					{0/1,1/0,1/1,2/1,2/2,3/2,4/2}
					{
						\cell{\x}{\y}
					}
				\end{scope}	
				
			\end{tikzpicture}
\end{subfigure}
\hspace{0.000005cm}
\begin{subfigure}{.20\linewidth}
\begin{tikzpicture}[scale=.4,rotate=0]
			\begin{scope}[xshift = 1cm]
					\foreach \x/\y in 
					{0/0,1/0,1/1,2/0,2/-1,3/-1,4/-1}
					{
						\cell{\x}{\y}
					}
				\end{scope}	
				
			\end{tikzpicture}
\end{subfigure} 
\par\bigskip
\par\bigskip
\begin{subfigure}{.20\linewidth}
\begin{tikzpicture}[scale=.4,rotate=0]
			\begin{scope}[xshift = 1cm]
					\foreach \x/\y in 
				{0/0,1/0,1/1,2/0,2/0,3/0,4/0}
					{
						\cell{\x}{\y}
					}
				\end{scope}	
				
			\end{tikzpicture}
\end{subfigure} \hspace{0.000005cm}
\begin{subfigure}{.20\linewidth}
\begin{tikzpicture}[scale=.4,rotate=0]
			\begin{scope}[xshift = 1cm]
					\foreach \x/\y in 
					{0/1,1/0,1/1,2/0,2/-1,3/-1,4/-1}
					{
						\cell{\x}{\y}
					}
				\end{scope}	
				
			\end{tikzpicture}
\end{subfigure}
\hspace{0.000005cm}
\begin{subfigure}{.20\linewidth}
\begin{tikzpicture}[scale=.4,rotate=0]
			\begin{scope}[xshift = 1cm]
					\foreach \x/\y in 
					{0/1,1/0,1/1,2/0,2/-1,3/0,4/0}
					{
						\cell{\x}{\y}
					}
				\end{scope}	
				
			\end{tikzpicture}
\end{subfigure}
  \caption{Insertion of two cells in successive columns.}
  \label{a3}
\end{figure}
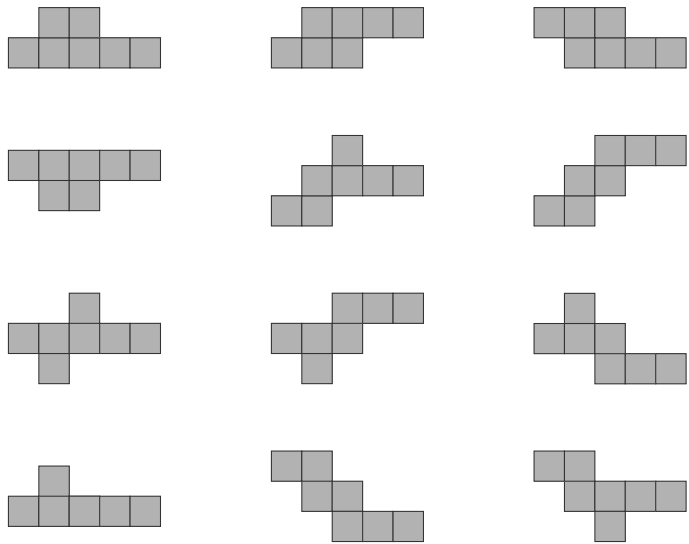

			\end{itemize}
		\item If we add two cells in two non-successive columns, we have the following subcases:
			\begin{itemize}
			\item[$\blacktriangleright$] One cell is added on the first column and the other one the last column. In this subcase, there is $4$ possibilities (see Fig. \ref{a5}).
			\begin{figure}[H]
\centering
 \begin{subfigure}{.20\linewidth}
\begin{tikzpicture}[scale=.4,rotate=0]
			\begin{scope}[xshift = 1cm]
					\foreach \x/\y in 
					{0/0,0/1,1/0,2/0,3/0,4/1,4/0}
					{
						\cell{\x}{\y}
					}
				\end{scope}	
							\end{tikzpicture}
\end{subfigure} \hspace{0.000005cm}
\begin{subfigure}{.20\linewidth}
\begin{tikzpicture}[scale=.4,rotate=0]
			\begin{scope}[xshift = 1cm]
					\foreach \x/\y in 
				{0/0,0/1,1/0,2/0,3/0,4/-1,4/0}
					{
						\cell{\x}{\y}
					}
				\end{scope}	
				
			\end{tikzpicture}
\end{subfigure} 
\par\bigskip
\par\bigskip
\begin{subfigure}{.20\linewidth}
\begin{tikzpicture}[scale=.4,rotate=0]
			\begin{scope}[xshift = 1cm]
					\foreach \x/\y in 
					{0/0,0/-1,1/0,2/0,3/0,4/1,4/0}
					{
						\cell{\x}{\y}
					}
				\end{scope}	
				
			\end{tikzpicture}
\end{subfigure} \hspace{0.000005cm}
\begin{subfigure}{.20\linewidth}
\begin{tikzpicture}[scale=.4,rotate=0]
			\begin{scope}[xshift = 1cm]
					\foreach \x/\y in 
				{0/0,0/-1,1/0,2/0,3/0,4/-1,4/0}
					{
						\cell{\x}{\y}
					}
				\end{scope}	
				
			\end{tikzpicture}
\end{subfigure}
\caption{Insertion of two cells in the first and last columns.}
  \label{a5} 
\end{figure}
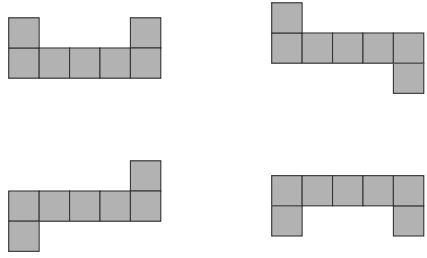
			\item[$\blacktriangleright$] If one cell is added on the first column and the other one on one of the columns in the middle different from the second, we have $8(k-3)$ possibilities, $2$ for adding a cell on the first column and for each adding we have $4(k-3)$ possibilities of adding a cell on a middle column different from the second (see Fig. \ref{a6}).
			\begin{figure}[H]
\centering
  
\begin{subfigure}{.20\linewidth}
\begin{tikzpicture}[scale=.4,rotate=0]
			\begin{scope}[xshift = 6cm]
					\foreach \x/\y in 
					{0/0,0/1,1/0,2/0,3/0,3/1,4/0}
					{
						\cell{\x}{\y}
					}
				\end{scope}	
				
			\end{tikzpicture}
\end{subfigure} \hspace{0.000005cm}
\begin{subfigure}{.20\linewidth}
\begin{tikzpicture}[scale=.4,rotate=0]
			\begin{scope}[xshift = 6cm]
					\foreach \x/\y in 
					{0/0,0/1,1/0,2/0,3/0,3/1,4/1}
					{
						\cell{\x}{\y}
					}
				\end{scope}	
				
			\end{tikzpicture}
\end{subfigure} \hspace{0.000005cm}
\begin{subfigure}{.20\linewidth}
\begin{tikzpicture}[scale=.4,rotate=0]
			\begin{scope}[xshift = 6cm]
					\foreach \x/\y in 
					{0/0,0/1,1/0,2/0,3/-1,3/0,4/-1}
					{
						\cell{\x}{\y}
					}
				\end{scope}	
				
			\end{tikzpicture}
\end{subfigure}

\par\bigskip
\par\bigskip
\begin{subfigure}{.20\linewidth}
\begin{tikzpicture}[scale=.4,rotate=0]
			\begin{scope}[xshift = 6cm]
					\foreach \x/\y in 
					{0/0,0/1,1/0,2/0,3/-1,3/0,4/-0}
					{
						\cell{\x}{\y}
					}
				\end{scope}	
				
			\end{tikzpicture}
\end{subfigure} \hspace{0.000005cm}
\begin{subfigure}{.20\linewidth}
\begin{tikzpicture}[scale=.4,rotate=0]
			\begin{scope}[xshift = 6cm]
					\foreach \x/\y in 
					{0/0,0/-1,1/0,2/0,3/0,3/1,4/0}
					{
						\cell{\x}{\y}
					}
				\end{scope}	
				
			\end{tikzpicture}
\end{subfigure} \hspace{0.000005cm}
\begin{subfigure}{.20\linewidth}
\begin{tikzpicture}[scale=.4,rotate=0]
			\begin{scope}[xshift = 6cm]
					\foreach \x/\y in 
					{0/0,0/-1,1/0,2/0,3/0,3/1,4/1}
					{
						\cell{\x}{\y}
					}
				\end{scope}	
				
			\end{tikzpicture}
\end{subfigure}
\par\bigskip
\par\bigskip
\begin{subfigure}{.20\linewidth}
\begin{tikzpicture}[scale=.4,rotate=0]
			\begin{scope}[xshift = 6cm]
					\foreach \x/\y in 
					{0/0,0/-1,1/0,2/0,3/-1,3/0,4/-1}
					{
						\cell{\x}{\y}
					}
				\end{scope}	
				
			\end{tikzpicture}
\end{subfigure}\hspace{0.000005cm}
\begin{subfigure}{.20\linewidth}
\begin{tikzpicture}[scale=.4,rotate=0]
			\begin{scope}[xshift = 6cm]
					\foreach \x/\y in 
					{0/0,0/-1,1/0,2/0,3/-1,3/0,4/0}
					{
						\cell{\x}{\y}
					}
				\end{scope}	
				
			\end{tikzpicture}
\end{subfigure}
  \caption{Insertion of two cells in the first and a middle column.}
  \label{a6}
\end{figure}
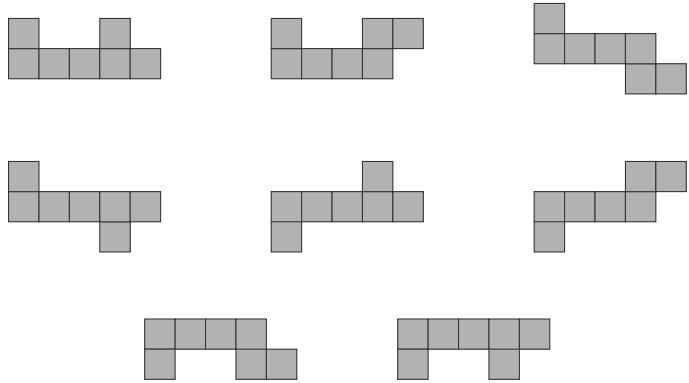
			\item[$\blacktriangleright$] One cell is added on the last column and the other on one of the columns in the middle different from the before last column. It is a similar to the previous subcase, so we have $8(k-3)$ possibilities.
			\item[$\blacktriangleright$] 
The subcase where the two cells are added on two column different from the first and the last ones, is equivalent to the tilling of a board of length $1\times (k-1)$ with two dominoes and $k-5$ squares (see Fig. \ref{tlp}).
\begin{figure}[H]
\centering
\begin{tikzpicture}[scale=.5,rotate=0]
   \foreach \x/\y in {0/0,1/0,2/0,3/0,4/0,5/0,6/0,7/0}{
     \cell{\x}{\y};
  }
   \fill[color=black!60](2,0)--(2,1)--(4,1)--(4,0)--cycle;
   \fill[color=black!60](5,0)--(5,1)--(7,1)--(7,0)--cycle;
   \draw(2,0)--(2,1)--(4,1)--(4,0)--cycle;
   \draw(5,0)--(5,1)--(7,1)--(7,0)--cycle;
   
   \begin{scope}[xshift = 10cm,yshift=0cm]
     \fill[color=black!50] (0,0) -- (2.8,0) -- (2.7,-.5) -- (3.5,.5) -- (2.7,1.5) -- (2.8,1) -- (0,1) -- cycle;
   \end{scope}
				
   \begin{scope}[xshift = 15cm]
     \foreach \x/\y in {0/0,1/0,2/0,2/1,3/0,4/0,5/0,5/1,6/0,7/0}{
     \cell{\x}{\y};
     }
   \end{scope}
 \end{tikzpicture}
 \caption{A tiling of a board and its equivalent polyomino.}
 \label{tlp}
\end{figure}
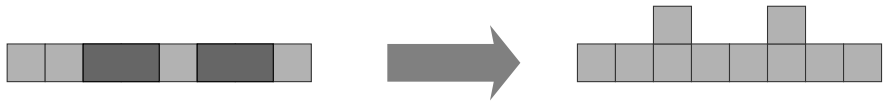
The number of $1\times n$ tilings using exactly $j$ dominoes is	
$$\binom{n-j}{j},\: (j=0,1,\cdots,\lfloor{n/2}\rfloor)\: \cite{refben}$$		
In this subcase $n=k-1$ and $j=2$. So we have $\binom{k-3}{2}$ possibilities of choosing two non-successive columns different from the first and the last ones. For each choice we have $16$ different constructions. Therefore, there are $16\binom{k-3}{2}$ possibilities (see Fig. \ref{a4}).		
			
			\begin{figure}[h]
\centering
 \begin{subfigure}{.20\linewidth}
\begin{tikzpicture}[scale=.4,rotate=0]
			\begin{scope}[xshift = 1cm]
					\foreach \x/\y in 
					{0/0,1/0,1/1,2/0,3/0,3/1,4/0}
					{
						\cell{\x}{\y}
					}
				\end{scope}	
				
			\end{tikzpicture}
\end{subfigure} \hspace{0.000005cm}
\begin{subfigure}{.20\linewidth}
\begin{tikzpicture}[scale=.4,rotate=0]
			\begin{scope}[xshift = 1cm]
					\foreach \x/\y in 
				{0/0,1/0,1/1,2/0,3/0,3/1,4/1}
					{
						\cell{\x}{\y}
					}
				\end{scope}	
				
			\end{tikzpicture}
\end{subfigure} \hspace{0.000005cm}
\begin{subfigure}{.20\linewidth}
\begin{tikzpicture}[scale=.4,rotate=0]
			\begin{scope}[xshift = 1cm]
					\foreach \x/\y in 
					{0/1,1/0,1/1,2/0,3/0,3/1,4/0}
					{
						\cell{\x}{\y}
					}
				\end{scope}	
				
			\end{tikzpicture}
\end{subfigure}
\par\bigskip
\par\bigskip
\begin{subfigure}{.20\linewidth}
\begin{tikzpicture}[scale=.4,rotate=0]
			\begin{scope}[xshift = 1cm]
					\foreach \x/\y in 
					{0/1,1/0,1/1,2/0,3/0,3/1,4/1}
					{
						\cell{\x}{\y}
					}
				\end{scope}	
				
			\end{tikzpicture}
\end{subfigure}
\hspace{0.000005cm}
\begin{subfigure}{.20\linewidth}
\begin{tikzpicture}[scale=.4,rotate=0]
			\begin{scope}[xshift = 1cm]
					\foreach \x/\y in 
					{0/0,1/0,1/1,2/1,3/0,3/1,4/0}
					{
						\cell{\x}{\y}
					}
				\end{scope}	
				
			\end{tikzpicture}
\end{subfigure} \hspace{0.000005cm}
\begin{subfigure}{.20\linewidth}
\begin{tikzpicture}[scale=.4,rotate=0]
			\begin{scope}[xshift = 1cm]
					\foreach \x/\y in 
				{0/0,1/0,1/1,2/1,3/0,3/1,4/1}
					{
						\cell{\x}{\y}
					}
				\end{scope}	
				
			\end{tikzpicture}
\end{subfigure} 
\par\bigskip
\par\bigskip
\begin{subfigure}{.20\linewidth}
\begin{tikzpicture}[scale=.4,rotate=0]
			\begin{scope}[xshift = 1cm]
					\foreach \x/\y in 
					{0/1,1/0,1/1,2/1,3/0,3/1,4/0}
					{
						\cell{\x}{\y}
					}
				\end{scope}	
				
			\end{tikzpicture}
\end{subfigure}
\hspace{0.000005cm}
\begin{subfigure}{.20\linewidth}
\begin{tikzpicture}[scale=.4,rotate=0]
			\begin{scope}[xshift = 1cm]
					\foreach \x/\y in 
					{0/1,1/0,1/1,2/1,3/0,3/1,4/1}
					{
						\cell{\x}{\y}
					}
				\end{scope}	
				
			\end{tikzpicture}
\end{subfigure}
\hspace{0.000005cm}
\begin{subfigure}{.20\linewidth}
\begin{tikzpicture}[scale=.4,rotate=0]
			\begin{scope}[xshift = 1cm]
					\foreach \x/\y in 
					{0/0,1/0,1/1,2/1,3/1,3/2,4/1}
					{
						\cell{\x}{\y}
					}
				\end{scope}	
				
			\end{tikzpicture}
\end{subfigure}
\par\bigskip
\par\bigskip
\begin{subfigure}{.20\linewidth}
\begin{tikzpicture}[scale=.4,rotate=0]
			\begin{scope}[xshift = 1cm]
					\foreach \x/\y in 
				{0/0,1/0,1/1,2/1,3/1,3/2,4/2}
					{
						\cell{\x}{\y}
					}
				\end{scope}	
				
			\end{tikzpicture}
\end{subfigure} \hspace{0.000005cm}
\begin{subfigure}{.20\linewidth}
\begin{tikzpicture}[scale=.4,rotate=0]
			\begin{scope}[xshift = 1cm]
					\foreach \x/\y in 
					{0/1,1/0,1/1,2/1,3/1,3/2,4/1}
					{
						\cell{\x}{\y}
					}
				\end{scope}	
				
			\end{tikzpicture}
\end{subfigure}
\hspace{0.000005cm}
\begin{subfigure}{.20\linewidth}
\begin{tikzpicture}[scale=.4,rotate=0]
			\begin{scope}[xshift = 1cm]
					\foreach \x/\y in 
					{0/1,1/0,1/1,2/1,3/1,3/2,4/2}
					{
						\cell{\x}{\y}
					}
				\end{scope}	
				
			\end{tikzpicture}
\end{subfigure}
\par\bigskip
\par\bigskip
\begin{subfigure}{.20\linewidth}
\begin{tikzpicture}[scale=.4,rotate=0]
			\begin{scope}[xshift = 1cm]
					\foreach \x/\y in 
					{0/2,1/2,1/3,2/2,3/1,3/2,4/1}
					{
						\cell{\x}{\y}
					}
				\end{scope}	
				
			\end{tikzpicture}
\end{subfigure}\hspace{0.000005cm}
\begin{subfigure}{.20\linewidth}
\begin{tikzpicture}[scale=.4,rotate=0]
			\begin{scope}[xshift = 1cm]
					\foreach \x/\y in 
				{0/2,1/2,1/3,2/2,3/1,3/2,4/2}
					{
						\cell{\x}{\y}
					}
				\end{scope}	
				
			\end{tikzpicture}
\end{subfigure}\hspace{0.000005cm}
\begin{subfigure}{.20\linewidth}
\begin{tikzpicture}[scale=.4,rotate=0]
			\begin{scope}[xshift = 1cm]
					\foreach \x/\y in 
					{0/3,1/2,1/3,2/2,3/1,3/2,4/1}
					{
						\cell{\x}{\y}
					}
				\end{scope}	
				
			\end{tikzpicture}
\end{subfigure}
\par\bigskip
\par\bigskip
\begin{subfigure}{.20\linewidth}
\begin{tikzpicture}[scale=.4,rotate=0]
			\begin{scope}[xshift = 1cm]
					\foreach \x/\y in 
					{0/3,1/2,1/3,2/2,3/1,3/2,4/2}
					{
						\cell{\x}{\y}
					}
				\end{scope}	
				
			\end{tikzpicture}
\end{subfigure}

  \caption{Insertion of two cells in non-successive columns.}
  \label{a4}
\end{figure}
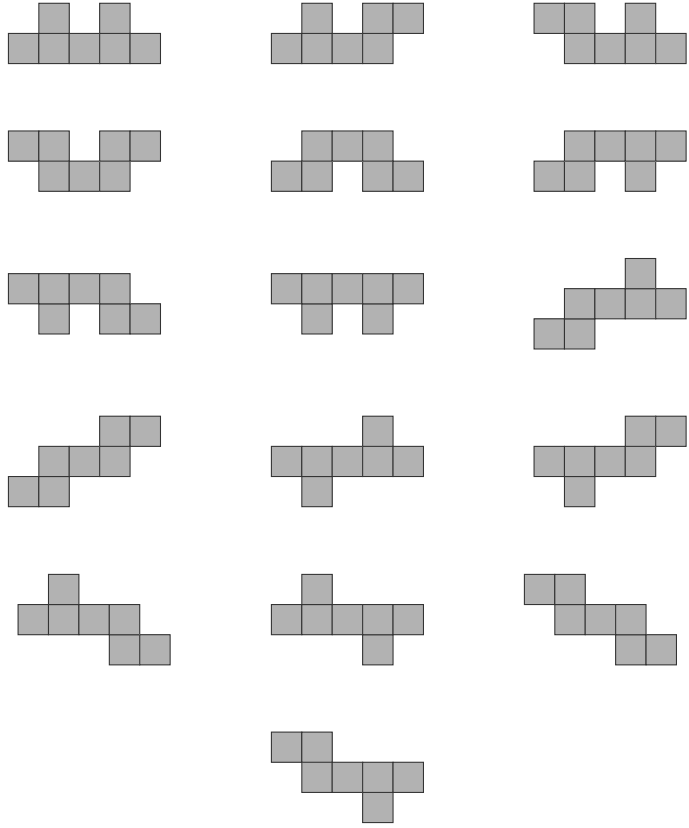			

			\end{itemize}

		\end{itemize}
\end{itemize}
By summing the number of possibilities, we obtain $8k^2-19k+16$.

\end{proof}

In general case, we have the following result,
\begin{theorem}
\label{t5_2}
$h_{k,k+i}$ is a polynomial of degree $i$ whose highest-degree-term is equal to $\cfrac{4^i}{i !}$ for $k\geq i+1$.
\end{theorem}
\begin{proof}
The proof is based on the same principle as in Theorem \ref{t5_1}. Let us consider the polyomino having an area $k$ and $k$ columns and let us add to it $i$ cells. If we add $i$ cells on $i$ non-successive columns different from the first and the last columns, it is equivalent to find the number of tilings of a board of length $1\times (k-1)$ with $i$ dominoes and $k-2i-1$ squares. Thus, we have $\binom{k-i-1}{i}$ different choosing. As, there are $4^i$ possibilities for each one, we obtain $4^i\binom{k-i-1}{i}$ possible constructions. The expression $4^i\binom{k-i-1}{i}$ is a polynomial of degree $i$ in $k$ with the coefficient of $k^i$ equals to $\cfrac{4^i}{i!}$. In other cases of choosing $i$ columns with at least two successive columns, we have $\binom{k-2}{i}-\binom{k-i-1}{i}$ possibilities, with $\binom{k-2}{i}$ the number of choosing $i$ columns different from the first and the last ones. The expression $\binom{k-2}{i}-\binom{k-i-1}{i}$ is a polynomial of degree $i-1$. In all other cases, we choose at most $i-1$ columns, leading to a term of degree lower then $i$. Therefore we obtain the result.

\end{proof}

From Theorem \ref{t5_2} and the first values of $h_{k,n}$, we get the following corollary,
\begin{cor}
\begin{align*}
h_{k,k+3}&=\cfrac{32}{3}k^3-44k^2+\cfrac{268}{3}k-76 \hspace{6pt}\text{with}\hspace{6pt} k\geq 4.\\
h_{k,k+4}&=\cfrac{32}{3}k^4-\cfrac{200}{3}k^3+\cfrac{1403}{6}k^2-\cfrac{2717}{6}k+384 \hspace{6pt}\text{with} \hspace{6pt}k\geq 5.\\
h_{k,k+5}&=\cfrac{128}{15}k^5-\cfrac{224}{3}k^4+\cfrac{1174}{3}k^3-\cfrac{3784}{3}k^2+\cfrac{35522}{15}k-2004 \hspace{6pt}\text{with} \hspace{6pt}k\geq 6.\\
h_{k,k+6}&=\cfrac{256}{45}k^6-\cfrac{992}{15}k^5+\cfrac{4292}{9}k^4-\cfrac{13427}{6}k^3+\cfrac{617753}{90}k^2-\cfrac{189503}{15}k+10672\\
& \hspace{6pt}\text{with} \hspace{6pt}k\geq 7.\\
\end{align*}
\end{cor}

\subsection{Plateau polycubes}

In the case of polycubes, the reasoning is similar to the one used for polyominoes.\\
Let $r_{k,m}$ be the number of plateau polycubes of width $k$ and lateral area $m$.
\begin{theorem}\mbox{}
\begin{enumerate}
\item $r_{k,2k}=1$, with $k\geq 1$.
\item $r_{k,2k+1}=8k-8$, with $k\geq 2$.
\item $r_{k,2k+2}=32k^2-70k+48$, with $k\geq 3$.
\end{enumerate}
\end{theorem}

\begin{proof}
The proof of these formulas is based on a unique principle. To build a plateau polycube of width $k$ and lateral area $2k+i$, we start from two polyominoes having area $k$ and $k$ columns. Next, we enumerate all the ways of adding $i$ cells on the two polyominoes.  For each couple of column-convex polyominoes we associate a unique plateau polycube. We only detail the proof of the value $r_{k,2k+2}$. \\
Let us consider two polyominoes of $k$ columns and area $k$ and let us enumerate all the ways to inject two cells in the polyominoes to obtain two column-convex polyominoes.   
We have the following cases:
\begin{itemize}
\item One cell is added in the first polyomino and the other in the second, the same way as in Theorem \ref{t5_1}. We have the $4k-4$ possibilities of construction for each polyomino. So, we have $(4k-4)^2$ possibilities.
\item The two cells are added in one rectangle. This case is enumerated in Theorem \ref{t5_1}. Thus, we have $2\times(8k^2-19k+16)$ possibilities.
\end{itemize}
By summing all the cases we have the result.
\end{proof}
\begin{theorem}
\label{t5_4}
$r_{k,2k+i}$ is a polynomial of degree $i$ whose highest-degree-term is equal to $\cfrac{8^i}{i !}$ for $k\geq i$.
\end{theorem}
\begin{proof}
To build a plateau polycube of width $k$ and lateral area $2k+i$, we start with two polyominoes having an area $k$ and $k$ columns, than we enumerate all the ways of adding $i$ cells in one or the two rectangles to form two column-convex polyominoes. For each couple of column-convex polyominoes we associate a unique plateau polycube.\\
If we add $j$ cells on the first polyomino, then the number of possibilities is a polynomial of degree $j$ with the coefficient of $k^j$ equal to $\cfrac{4^j}{j!}$. Thus,  we add $i-j$ cells on the second polyomino, the number of possibilities is a polynomial of degree $i-j$ with the coefficient of $k^{i-j}$ equal to $\cfrac{4^{i-j}}{(i-j)!}$. So, the number of possibilities of adding $i$ cells in the two polyominoes is a polynomial of degree $i$ with the coefficient of the $k^i$ is $\cfrac{4^i}{j!(i-j)!}$.\\
By summing the coefficient of $k^i$ for all possible $j's$ we obtain,
$$\sum\limits_{j=0}^{i}\frac{4^i}{j!(i-j)!}=4^i\sum\limits_{j=0}^{i}\frac{1}{j!(i-j)!}=\frac{4^i}{i!}\sum\limits_{j=0}^{i}\frac{i!}{j!(i-j)!}=\frac{4^i 2^i}{i!}= \frac{8^i}{i !}.$$
\end{proof}

As in previous subsection, using the result of Theorem \ref{t5_4} and the first values of $r_{k,m}$, we get the following corollary,
\begin{cor}
\begin{align*}
r_{k,k+3}&=\cfrac{256}{3}k^3-304k^2+\cfrac{1376}{3}k-280 \hspace{6pt}\text{with}\hspace{6pt} k\geq 4.\\
r_{k,k+4}&=\cfrac{512}{3}k^4-\cfrac{2624}{3}k^3+\cfrac{6454}{3}k^2-\cfrac{8509}{3}k+1632 \hspace{6pt}\text{with}\hspace{6pt} k\geq 5.\\
r_{k,k+5}&=\cfrac{4096}{15}k^5-\cfrac{5632}{3}k^4+\cfrac{19888}{3}k^3-\cfrac{42104}{3}k^2+\cfrac{85888}{5}k-9512 \hspace{6pt}\text{with}\hspace{6pt} k\geq 6.\\
r_{k,k+6}&=\cfrac{16384}{45}k^6-\cfrac{48128}{15}k^5+\cfrac{136256}{9}k^4-45444k^3+\cfrac{3971986}{45}k^2-\cfrac{1543582}{15}k+\\
&55440 \hspace{6pt}\text{with}\hspace{6pt} k\geq 7.\\
\end{align*}
\end{cor}

\end{document}